\documentclass{amsart}

\usepackage{amsthm,amssymb,amsmath,amscd}

\theoremstyle{plain}
\newtheorem{thm}{Theorem}[section]
\newtheorem{lem}[thm]{Lemma}
\newtheorem{prop}[thm]{Proposition}

\theoremstyle{definition}
\newtheorem*{defn}{Definition}
\newtheorem*{ex}{Example}

\theoremstyle{remark}
\newtheorem*{rem}{Remark}

\title[C$^*$-algebras generated by composition operators]
{C$^*$-algebras generated by multiplication operators and composition operators by functions with self-similar branches {\rm II}}
\author{Hiroyasu Hamada}

\address{National Institute of Technology (KOSEN), Sasebo College, 
Okishin, Sasebo, Nagasaki, 857-1193, Japan.}
\email{h-hamada@sasebo.ac.jp}



\begin{document}

\begin{abstract}
Let $K$ be a compact metric space and let $\varphi: K \to K$ be continuous.
We study a C$^*$-algebra $\mathcal{MC}_\varphi$
generated by
all multiplication operators by continuous functions on $K$
and a composition operator $C_\varphi$ induced by $\varphi$ on
a certain $L^2$ space.
Let $\gamma = (\gamma_1, \dots, \gamma_n)$ be a system of proper contractions
on $K$. Suppose that $\gamma_1, \dots, \gamma_n$ are inverse branches of
$\varphi$ and $K$ is self-similar.
We consider the Hutchinson
measure $\mu^H$ of $\gamma$ and the $L^2$ space $L^2(K, \mu^H)$.
Then we show
that the C$^*$-algebra $\mathcal{MC}_\varphi$ is isomorphic to
the C$^*$-algebra $\mathcal{O}_\gamma (K)$ associated
with $\gamma$ under the open set condition
and the measure separation condition.
This is a generalization of our previous work, in which we studied the case
where $\gamma$ satisfied the finite branch condition.
\end{abstract}

\maketitle

\section{Introduction}

Several authors considered
C$^*$-algebras generated by composition operators
(and Toeplitz operators) on the Hardy space $H^2(\mathbb{D})$ on the open
unit disk
$\mathbb{D}$ (\cite{H1, HW, J1, J2, KMM1, KMM3, KMM2, Pa, Q, Q2, SA}).
On the other hand, there are some studies on C$^*$-algebras generated by
composition operators on $L^2$ spaces, for example \cite{H2, H3, M}.
Matsumoto \cite{M} introduced some
C$^*$-algebras associated with cellular automata generated by composition
operators and multiplication operators.

Let $R$ be a rational function of degree at least two,
let $J_R$ be the Julia set of $R$ and let $\mu^L$ be the Lyubich measure
of $R$. In \cite{H2}, we studied the C$^*$-algebra $\mathcal{MC}_R$
generated by
all multiplication operators by continuous functions in $C(J_R)$
and the composition operator $C_R$ induced by $R$
on $L^2(J_R, \mu^L)$.
We showed that the C$^*$-algebra $\mathcal{MC}_R$ is isomorphic to
the C$^*$-algebra $\mathcal{O}_R (J_R)$ associated with the complex dynamical
system $\{R^{\circ n} \}_{n=1} ^\infty$ introduced in \cite{KW1}.

Let $(K, d)$ be a compact metric space, 
let $\gamma = (\gamma_1, \dots, \gamma_n)$ be a system of proper contractions
on $K$ and let $\varphi: K \to K$ be continuous.
Suppose that $\gamma_1, \dots, \gamma_n$ are inverse branches of
$\varphi$ and  $K$ is self-similar. Let $\mu^H$ be the Huchinson measure of
$\gamma$. In \cite{H3}, we studied the C$^*$-algebra $\mathcal{MC}_\varphi$
generated by
all multiplication operators by continuous functions in $C(K)$
and the composition operator $C_\varphi$ induced by $\varphi$
on $L^2(K, \mu^H)$.
Assume that the system
$\gamma = (\gamma_1, \dots, \gamma_n)$ satisfies the open set condition,
the finite branch condition
and the measure separation condition in $K$. 
We showed that  the C$^*$-algebra $\mathcal{MC}_\varphi$ is isomorphic to
the C$^*$-algebra $\mathcal{O}_\gamma (K)$ associated
with $\gamma$ introduced in \cite{KW2}.

In this paper we consider a generalization of \cite{H3}.
We can remove the second condition
that the system $\gamma$ satisfy the finite branch
condition.
We also show that
$\mathcal{MC}_\varphi$ is isomorphic to $\mathcal{O}_\gamma (K)$
associated with $\gamma$. In this proof, we do not use a countable basis of
a Hilbert bimodule.

There is an applications of the main theorem.
Let $\tau$ be the tent map $\tau: [0,1] \to [0,1]$ 
defined by
\[
  \tau (x) = \begin{cases}
                 2x & 0 \leq x \leq \frac{1}{2}, \\
                 -2x + 2 & \frac{1}{2} \leq x \leq 1
                \end{cases}
\]
and let $\varphi$ be the map $\varphi: [0,1] \times [0, 1] \to
[0,1] \times [0, 1]$ defined by
$\varphi(x, y) = (\tau(x), \tau(y))$ for $x, y \in [0, 1]$.
Suppose that $\gamma_1, \gamma_2, \gamma_3, \gamma_4$
are inverse branches of $\varphi$.
Since the system $\gamma  = (\gamma_1, \gamma_2, \gamma_3, \gamma_4)$
does not satisfy
the finite branch condition, we cannot adapt \cite[Theorem 4.4]{H3}.
However we can adapt this main theorem in this paper and the C$^*$-algebra
$\mathcal{MC}_\varphi$ is isomorphic to the Cuntz algebra
$\mathcal{O}_\infty$.

\section{Covariant relations}

Let $(K,d)$ be a compact metric space. A continuous map $\gamma: K \to K$
is called a {\it proper contraction} if there exists constants
$0 < c_1 \leq c_2 < 1$ such that
\[
   c_1 d(x, y) \leq d(\gamma(x), \gamma(y)) \leq c_2 d(x, y), \quad x, y \in K.
\]

Let $\gamma = (\gamma_1, \dots, \gamma_n)$ be a family of proper contractions
on $(K,d)$. We say that $K$ is called {\it self-similar}
with respect to $\gamma$ if $K = \bigcup_{i = 1} ^n \gamma_i (K)$.
See \cite{F} and \cite{Ki} for more on fractal sets.
We say that $\gamma$ satisfies the {\it open set condition} in $K$
if there exists a non-empty open set $V \subset K$ such that
\[
   \bigcup_{i = 1} ^n \gamma_i(V) \subset V \quad \text{and} \quad
   \gamma_i (V) \cap \gamma_j (V) = \emptyset \quad \text{for} \quad
   i \neq j.
\]

For a system $\gamma$ of proper contractions on a compact metric space $K$,
we introduce the following subsets of $K$.
\begin{align*}
B_\gamma &= \{ y \in K \, | \, y = \gamma_i (x) = \gamma_j (x)
\, \, \text{for some} \, \, x \in K \, \,  \text{and} \, \, i \neq j \}, \\
C_\gamma &= \{ x \in K \, | \, \gamma_i (x) = \gamma_j (x) \, \, 
\text{for some} \, \, i \neq j \}.
\end{align*}
We say that $\gamma$ satisfies the {\it finite branch condition} if $C_\gamma$
is finite set.

Let us denote by $\mathcal{B}(K)$ the Borel $\sigma$-algebra on $K$.

\begin{lem}[\cite{Hu}] \label{lem:Hutchinson measure}
Let $K$ be a compact metric space and let $\gamma$ be a system of proper
contractions. If $p_1, \dots, p_n \in \mathbb{R}$ satisfy
$\sum_{i=1} ^n p_i = 1$ and $p_i > 0$ for $i$,
then there exists a unique probability measure $\mu$ on $K$ such that
\[
   \mu(E) = \sum_{i=1}^n  p_i \mu (\gamma_i ^{-1}(E))
\]
for $E \in \mathcal{B}(K)$.
\end{lem}

We call the measure $\mu$ given by Lemma \ref{lem:Hutchinson measure}
the {\it self-similar measure} on $K$ with $\{p_i \}_{i = 1} ^n$.
In particular, we denote by $\mu^H$ the self-similar measure with
$p_i = \frac{1}{n}$ for $i$ and call this measure the {\it Hutchinson measure}.
We say that $\gamma$ satisfies the {\it measure separation condition} in $K$
if $\mu (\gamma_i (K) \cap \gamma_j (K)) = 0$ for any self-similar measure $\mu$ and $i \neq j$.

For $a \in L^\infty (K, \mathcal{B}(K), \mu^H)$, we define
the multiplication operator $M_a$ on
$L^2 (K, \mathcal{B}(K), \mu^H)$ by
$M_a f = a f$ for $f \in L^2 (K, \mathcal{B}(K), \mu^H)$.
Let $\varphi : K \to K$ be measureable.
Suppose that $\gamma_1, \dots, \gamma_n$ are inverse
branches of $\varphi$, that is, $\varphi (\gamma_i (x)) = x$
for $x \in K$ and $i = 1, \dots, n$.
We define $C_\varphi f = f \circ \varphi$
for $f \in L^2 (K, \mathcal{B}(K)$.
For $f \in L^1 (K, \mathcal{B}(K), \mu^H)$, we define a function $\mathcal{L}_\varphi f: K \to \mathbb{C}$
by
\[
   (\mathcal{L}_\varphi f)(x) = \frac{1}{n} \sum_{i = 1} ^n f(\gamma_i (x)),
   \quad x \in K.
\]
For $f \in C(K)$, we can easily see that $\mathcal{L}_\varphi f \in C(K)$ since $\gamma_1, \dots, \gamma_n$ are continuous functions.

\begin{prop}[{\cite[Proposition 2.5]{H3}}] \label{prop:covariant}
Let $\gamma = (\gamma_1, \dots, \gamma_n)$ be a system of proper contractions.
Assume that $K$ is self-similar and
the system $\gamma = (\gamma_1, \dots, \gamma_n)$ satisfies the
measure separation condition in $K$.
Then $C_\varphi$ is an isometry on $L^2 (K, \mathcal{B}(K), \mu^H)$,
$\mathcal{L}_\varphi$ is bounded on $L^\infty (K, \mathcal{B}(K), \mu^H)$ and
\[
  C_\varphi ^* M_a C_\varphi = M_{\mathcal{L}_\varphi (a)}
\]
for $a \in L^\infty (K, \mathcal{B}(K), \mu^H)$.
\end{prop}

The operator $C_\varphi$ is called
the {\it composition operator}
on $L^2 (K, \mathcal{B}(K), \mu^H)$ induced by $\varphi$.

\section{C$^*$-algebras associated with self-similar sets}

We recall the construction of Cuntz-Pimsner algebras \cite{Pi} (see also
\cite{K}). 
Let $A$ be a C$^*$-algebra and let $X$ be a right Hilbert $A$-module.
We denote by $\mathcal{L}(X)$ the C$^*$-algebra of the adjointable bounded operators 
on $X$.  
For $\xi$, $\eta \in X$, the operator $\theta _{\xi,\eta}$
is defined by $\theta _{\xi,\eta}(\zeta) = \xi \langle \eta, \zeta \rangle_A$
for $\zeta \in X$. 
The closure of the linear span of these operators is denoted by $\mathcal{K}(X)$. 
We say that 
$X$ is a {\it Hilbert bimodule} (or {\it C$^*$-correspondence}) 
over $A$ if $X$ is a right Hilbert $A$-module 
with a $*$-homomorphism $\phi : A \rightarrow \mathcal{L}(X)$.
We always assume that $\phi$ is injective. 

A {\it representation} of the Hilbert bimodule $X$
over $A$ on a C$^*$-algebra $D$
is a pair $(\rho, V)$ constituted by a $*$-homomorphism $\rho: A \to D$ and
a linear map $V: X \to D$ satisfying
\[
  \rho(a) V_\xi = V_{\phi(a) \xi}, \quad
  V_\xi ^* V_\eta = \rho( \langle \xi, \eta \rangle_A)
\]
for $a \in A$ and $\xi, \eta \in X$.
It is known that $V_\xi \rho(b) = V_{\xi b}$ follows
automatically (see for example \cite{K}).
We define a $*$-homomorphism $\psi_V : \mathcal{K}(X) \to D$
by $\psi_V ( \theta_{\xi, \eta}) = V_{\xi} V_{\eta}^*$ for $\xi, \eta \in X$
(see for example \cite[Lemma 2.2]{KPW}).
A representation $(\rho, V)$ is said to be {\it covariant} if
$\rho(a) = \psi_V(\phi(a))$ for all $a \in J(X)
:= \phi ^{-1} (\mathcal{K}(X))$.

Let $(i, S)$ be the representation of $X$ which is universal for
all covariant representations. 
The {\it Cuntz-Pimsner algebra} ${\mathcal O}_X$ is 
the C$^*$-algebra generated by $i(a)$ with $a \in A$ and 
$S_{\xi}$ with $\xi \in X$.
We note that $i$ is known to be injective
\cite{Pi} (see also \cite[Proposition 4.11]{K}).
We usually identify $i(a)$ with $a$ in $A$.

Let $\gamma = (\gamma_1, \dots, \gamma_n)$ be a system of proper contractions
on a compact metric space $K$.
Assume that $K$ is self-similar.
Let $A = C(K)$ and $Y = C(\mathcal{C})$,
where $\mathcal{C} = \bigcup_{i = 1} ^n \{ (\gamma_i(y), y) \, | \, y
\in K \}$  is the cograph of $\gamma_i$.
Then $Y$ is an $A$-$A$ bimodule over $A$ by 
\[
  (a \cdot f \cdot b)(\gamma_i(y),y) = a(\gamma_i(y)) f(\gamma_i(y),y) b(y),\quad a, b \in A, \, 
  f \in Y.
\]
We define an $A$-valued inner product $\langle \ , \ \rangle_A$ on $Y$ by 
\[
  \langle f, g \rangle_A (y) = \sum _{i=1} ^n
  \overline{f(\gamma_i (y), y)} g(\gamma_i(y), y),
  \quad f, g \in Y, \, y \in K.
\]
Then $Y$ is a Hilbert bimodule over $A$. 
The C$^*$-algebra ${\mathcal O}_\gamma (K)$
is defined as the Cuntz-Pimsner algebra of the Hilbert bimodule 
$Y = C(\mathcal{C})$ over $A = C(K)$.

\section{Main theorem}

\begin{defn}
Let $\varphi : K \to K$ be continuous.
Suppose that composition operator $C_\varphi$ on
$L^2(K, \mathcal{B}(K), \mu^{H})$ is bounded.
We denote by $\mathcal{MC}_\varphi$ the C$^*$-algebra
generated by all multiplication operators by
continuous functions in $C(K)$
and the composition operator $C_\varphi$ on $L^2(K, \mathcal{B}(K), \mu^{H})$.
\end{defn}

Let $\varphi:K \to K$ be continuous.
Let $A = C(K)$ and $X = C(K)$.
Then $X$ is an $A$-$A$ bimodule over $A$ by 
\[
   (a \cdot \xi \cdot b) (x) = a(x) \xi(x) b(\varphi(x)) \quad a,b \in A, \, 
  \xi \in X.
\]
We define an $A$-valued inner product $\langle \ , \ \rangle_A$ on $X$ by
\[
  \langle \xi, \eta \rangle _A (x) = \frac{1}{n} \sum_{i = 1} ^n
  \overline{\xi(\gamma_i (x))} \eta(\gamma_i(x))
  \, \,
  \left (\, = (\mathcal{L}_\varphi (\overline{\xi} \eta))(x) \, \right ), \quad
  \xi, \eta \in X.
\]
Then $X$ is a Hilbert bimodule over $A$.
The left multiplication of $A$ on $X$ gives the left action
$\phi: A \to \mathcal{L}(X)$ such that $(\phi(a) \xi)(x) = a(x) \xi(x)$
for $a \in A$ and $\xi \in X$.
Let $\Phi: Y \to X$ be defined by
$(\Phi(f))(x) = \sqrt{n} f(x, \varphi(x))$ for $f \in Y$.
It is easy to see that $\Phi$ is an isomorphism and $X$ is
isomorphic to $Y$ as Hilbert bimodules over $A$.
Hence the C$^*$-algebra $\mathcal{O}_\gamma (K)$ is isomorphic to
the Cuntz-Pimsner algebra $\mathcal{O}_X$ constructed from $X$.

For $x \in K$, we define
\[
   I(x) = \{ i \in \{1, \dots, n \} \, | \, \text{there exists} \, \,
   y \in K \, \, \text{such that} \, \, x = \gamma_i(y) \}.
\]

\begin{lem}[{\cite[Lemma 2.2]{KW2}}] \label{lem:nbd}
Let $x \in K \smallsetminus B_\gamma$.
Then there exists an open neighbourhood $U_x$ of $x$
the following

\begin{enumerate}
\item $U_x
\cap B(\gamma_1, \dots, \gamma_n) = \emptyset$.

\item
If $i \in I(x)$, then $\gamma_j (\gamma_i^{-1}(U_x)) \cap U_x = \emptyset$ for
$j \neq i$.

\item
If $i \notin I(x)$, then $U_x \cap \gamma_i (K) = \emptyset$.
\end{enumerate}

\end{lem}

We now recall a description of the ideal $J(X)$ of $A$.
Assume that $\gamma = (\gamma_1, \dots, \gamma_n)$ satisfies the open set
condition in $K$.
By \cite[Proposition 2.6]{KW2}, we can write
$J(X) = \{ a \in A \, | \,
a \, \, \text{vanishes on} \, \, B_\gamma \}$.
We define a subset $J(X)^0$ of $J(X)$
by $J(X) ^0 = \{ a \in A \, | \,
a \, \, \text{vanishes on} \, \, B_\gamma \, \,
\text{and has compact support on } K \smallsetminus
B_\gamma \}$.
Then $J(X)^0$ is dense in $J(X)$.

\begin{lem} \label{lem:key_lem}
Assume that $\gamma = (\gamma_1, \dots, \gamma_n)$ satisfies the open set
condition in $K$.
Then, for $a \in J(X)^0$, there exists $\xi_1, \dots, \xi_m , \eta_1, \dots,
\eta_m \in X$ such that
\[
   \sum_{i = 1} ^m \theta_{\xi_i, \eta_i} = \phi(a).
\]
\end{lem}

\begin{proof}
For $x \in {\rm supp}(a)$, choose an open neighbourhood $U_x$
of $x$ as in Lemma \ref{lem:nbd}. By the same argument in the proof of
\cite[Proposition 2.4]{KW2}, we can choose $\{f_i \}_{i = 1} ^{m+1}$
in $A$ such that $0 \leq f_i \leq 1, \, {\rm supp}(f_i) \subset U_{x_i}$
for $i = 1, \dots, m$ and $\sum_{i = 1} ^m f_i(x) = 1$ for $x \in {\rm supp}(a)$.
Define $\xi_i, \eta_i \in X$ by $\xi_i(x) = n a(x) \sqrt{f_i (x)}$ and
$\eta_i (x) = \sqrt{f_i(x)}$. By a similar argument in the proof of
\cite[Proposition 2.4]{KW2}, we have
$\sum_{i = 1} ^m \theta_{\xi_i, \eta_i} = \phi(a)$.
\end{proof}

We regard the following lemma as
an operator version of Lemma \ref{lem:key_lem}.

\begin{lem} \label{lem:operator}
Let $\gamma = (\gamma_1, \dots, \gamma_n)$
be a system of proper contractions.
Assume that $K$ is self-similar and
the system $\gamma = (\gamma_1, \dots, \gamma_n)$ satisfies
the open set condition and
the measure separation condition in $K$.
Suppose that $\xi_1, \dots, \xi_m , \eta_1, \dots, \eta_m \in X$
are in Lemma \ref{lem:key_lem}.
Then, for $a \in J(X)^0$, we have
\[
   \sum_{i = 1} ^m M_{\xi_i} C_\varphi C_\varphi ^* M_{\eta_i} ^* = M_a.
\]
\end{lem}

\begin{proof}
Let $b \in C(K)$.
By Lemma \ref{lem:key_lem}, there exists
$\xi_1, \dots, \xi_m , \eta_1, \dots, \eta_m \in X$ such that
\[
   \sum_{i = 1} ^m \xi_i \cdot \langle \eta_i, b \rangle_A = a b.
\]
Since $b = M_b C_\varphi 1$, we have
\begin{align*}
\sum_{i = 1} ^m M_{\xi_i} C_\varphi C_\varphi ^* M_{\eta_i} ^* b
&= \sum_{i = 1} ^m M_{\xi_i} C_\varphi C_\varphi ^* M_{\eta_i} ^* M_b C_\varphi 1 \\
&= \sum_{i = 1} ^m M_{\xi_i} C_\varphi C_\varphi ^* M_{\overline \eta_i b} C_\varphi 1 \\
&= \sum_{i = 1} ^m M_{\xi_i} C_\varphi M_{\mathcal{L}_\varphi (\overline \eta_i b)} 1
   \quad \quad \text{by Proposition \ref{prop:covariant}} \\
&= \sum_{i = 1} ^m M_{\xi_i} M_{\mathcal{L}_\varphi (\overline \eta_i b) \circ \varphi} C_\varphi 1 \\
&= \sum_{i = 1} ^m \xi_i \mathcal{L}_\varphi (\overline \eta_i b) \circ \varphi \\
&= \sum_{i = 1} ^m \xi_i \cdot \langle \eta_i, b \rangle_A \\
&= M_a b.
\end{align*}
Since the Hutchinson measure $\mu^H$ on $K$ is regular, $C(K)$ is dense in
$L^2 (K, \mathcal{B}(K), \mu^H)$, the proof is complete.

\end{proof}

The following theorem is the main result of the paper.
This is a generalization of \cite[Theorem 4.4]{H3}.

\begin{thm} \label{thm:main}
Let $(K, d)$ be a compact metric space, 
let $\gamma = (\gamma_1, \dots, \gamma_n)$ be a system of proper contractions
on $K$ and let $\varphi: K \to K$ be continuous.
Suppose that $\gamma_1, \dots, \gamma_n$ are inverse branches of
$\varphi$. Assume that $K$ is self-similar and the system
$\gamma = (\gamma_1, \dots, \gamma_n)$ satisfies the open set condition
and the measure separation condition in $K$. 
Then $\mathcal{MC}_\varphi$ is isomorphic to $\mathcal{O}_\gamma (K)$.
\end{thm}

\begin{proof}
Put $\rho(a) = M_a$ and $V_\xi = M_\xi C_\varphi$ for $a \in A$ and
$\xi \in X$.
Then we can show $\rho(a) V_\xi = V_{a \cdot \xi}$ and
$V_\xi ^* V_\eta = \rho ( \langle \xi, \eta \rangle_A)$
for $a \in A$ and $\xi, \eta \in X$ as in the proof of
\cite[Theorem 4.4]{H3}.

Let $a \in J(X)^0$. By Lemma \ref{lem:key_lem},
there exits $\xi_1, \dots, \xi_m , \eta_1, \dots, \eta_m \in X$ such that
$\phi(a) = \sum_{i = 1} ^m \theta_{\xi_i, \eta_i}$.
From Lemma \ref{lem:operator} it follows that
\[
  \psi_V( \phi(a) )
  = \psi_V \left( \sum_{i = 1} ^m \theta_{\xi_i, \eta_i} \right)
  = \sum_{i = 1} ^m V_\xi V_{\eta_i} ^*
  = \sum_{i = 1} ^m M_{\xi_i} C_\varphi C_\varphi ^* M_{\eta_i} ^*
  = M_a
  = \rho(a).
\]
Since $J(X)^0$ is dense in $J(X)$, we have $\psi_V(\phi(a)) = \rho(a)$
for $a \in J(X)$.

By the universality and the simplicity of $\mathcal{O}_\gamma (K)$
(\cite[Theorem 3.8]{KW2}),
the C$^*$-algebra $\mathcal{MC}_\varphi$ is isomorphic to
$\mathcal{O}_\gamma (K)$.

\end{proof}

\section{Examples} \label{sec:Ex}

We give some examples for C$^*$-algebras generated by a composition
operator $C_\varphi$ and multiplication operators.

\begin{ex}
A tent map $\tau: [0,1] \to [0,1]$ is defined by
\[
  \tau (x) = \begin{cases}
                 2x & 0 \leq x \leq \frac{1}{2}, \\
                 -2x + 2 & \frac{1}{2} \leq x \leq 1.
                \end{cases}
\]
Let $\varphi: [0, 1] \times [0, 1] \to [0, 1] \times [0, 1]$ be given by
$\varphi(x, y) = (\tau(x), \tau(y))$ for $x, y \in [0,1]$.
Then inverse branches of $\varphi$ are $\gamma_1, \gamma_2, \gamma_3, \gamma_4$ such that
\begin{alignat*}{2}
\gamma_1 (x, y) &= \left( \frac{1}{2} x, \, \frac{1}{2} y \right), &\quad
\gamma_2 (x, y) &= \left( \frac{1}{2} x , \, - \frac{1}{2} y  + 1 \right), \\
\gamma_3 (x, y) &= \left( - \frac{1}{2} x + 1, \,  \frac{1}{2} y \right), &\quad
\gamma_4 (x, y) &= \left( - \frac{1}{2} x + 1, \, - \frac{1}{2} y  + 1 \right)
\end{alignat*} 
for $x, y \in [0, 1]$.
The maps $\gamma_1, \gamma_2, \gamma_3, \gamma_4$
are proper contractions and $K = [0, 1] \times [0, 1]$
is the self-similar set
with respect to $\gamma =  (\gamma_1, \gamma_2, \gamma_3, \gamma_4)$.
The system $\gamma$ satisfies
the open set condition in $K$.
Since
$C_\gamma = ([0, 1] \times \{ 1 \}) \cup (\{ 1 \} \times [0, 1])$,
the system $\gamma$ does not satisfy the finite branch condition.
The Hutchinson measure $\mu^H$ on $K$ coincides with
the Lebesgue measure $m$ on $K$. 
By \cite{Sc}, 
the system $\gamma$ satisfies the measure
separation condition in $K$.
We consider the composition operator $C_\varphi$ on
$L^2 (K, \mathcal{B}(K), m)$. By Theorem \ref{thm:main},
the C$^*$-algebra $\mathcal{MC}_\varphi$ is isomorphic to
$\mathcal{O}_\gamma (K)$.

Since
\[
  B_\gamma = \left( [0, 1] \times \left\{ \frac{1}{2} \right\} \right)
  \cup \left( \left\{ \frac{1}{2} \right\} \times [0, 1] \right),
\]
we have $K_0 (C(B_\gamma)) \cong \mathbb{Z}$ and
$K_1 (C(B_\gamma)) \cong 0$.
Since $K_0 (C(K)) \cong \mathbb{Z}$ and
$K_1 (C(K)) \cong 0$, it follows that
$K_0 (J(X)) \cong 0$ and $K_1 (J(X)) \cong 0$.
By the six-term exact sequence of the Cuntz-Pimsner algebra
$\mathcal{O}_\gamma (K)$ due to \cite{Pi}, we have
$K_0 (\mathcal{MC}_\varphi) \cong \mathbb{Z}$
, $K_1 (\mathcal{MC}_\varphi) \cong 0$ and
the unit $[I]_0$ in $K_0$ maps to $1$ in $\mathbb{Z}$.
Hence $\mathcal{MC}_\varphi$ is isomorphic to
the Cuntz algebra $\mathcal{O}_\infty$.
\end{ex}

\begin{ex}
A map $\sigma: [0,1] \to [0,1]$ is defined by
\[
  \sigma (x) = \begin{cases}
                 3x & 0 \leq x \leq \frac{1}{3}, \\
                 -3x + 2 & \frac{1}{3} \leq x \leq \frac{2}{3}, \\
                 3x - 2 & \frac{2}{3} \leq x \leq 1.
               \end{cases}
\]
Let $\varphi: [0, 1] \times [0, 1] \to [0, 1] \times [0, 1]$ be given by
$\varphi(x, y) = (\tau(x), \sigma(y))$ for $x, y \in [0,1]$.
Then inverse branches of $\varphi$ are $\gamma_1, \dots, \gamma_6$ such that
\begin{alignat*}{2}
\gamma_1 (x, y) &= \left( \frac{1}{2} x, \, \frac{1}{3} y \right), &\quad
\gamma_2 (x, y) &= \left( \frac{1}{2} x, \, - \frac{1}{3} y  + \frac{2}{3}\right), \\
\gamma_3 (x, y) &= \left( \frac{1}{2} x, \,  \frac{1}{3} y + \frac{2}{3} \right), &\quad
\gamma_4 (x, y) &= \left( - \frac{1}{2} x + 1, \, \frac{1}{3} y \right), \\
\gamma_5 (x, y) &= \left( - \frac{1}{2} x + 1, \,  - \frac{1}{3} y  + \frac{2}{3} \right), &\quad
\gamma_6 (x, y) &= \left( - \frac{1}{2} x + 1, \, \frac{1}{3} y  + \frac{2}{3} \right).
\end{alignat*}
for $x, y \in [0, 1]$.
The maps $\gamma_1, \dots, \gamma_6$
are proper contractions and $K = [0, 1] \times [0, 1]$
is the self-similar set
with respect to $\gamma =  (\gamma_1, \dots, \gamma_6)$.
The Hutchinson measure $\mu^H$ on $K$ coincides with
the Lebesgue measure $m$ on $K$.
The system $\gamma$ satisfies
the open set condition and the measure separation condition in $K$.
Similar to the previous example, we have
$(K_0(\mathcal{MC}_\varphi), [I]_0, K_1(\mathcal{MC}_\varphi)) \cong (\mathbb{Z}, 1, 0)$. Hence the C$^*$-algebra $\mathcal{MC}_\varphi$ is isomorphic to the Cuntz algebra $\mathcal{O}_\infty$.
\end{ex}

\begin{rem}
Let $\varphi: [0, 1] \to [0,1]$ be given by
$\varphi(x) = \tau(x)$ for $x \in [0, 1]$.
By \cite[Example]{H3}, the C$^*$-algebra $\mathcal{MC}_\varphi$
is isomorphic to the Cuntz algebra $\mathcal{O}_\infty$.
On the other hand, let $\psi: [0, 1] \to [0,1]$ be given by
$\psi(x) = \sigma(x)$ for $x \in [0, 1]$.
We have $K_0(\mathcal{MC}_\psi) \cong \mathbb{Z}^2$ and
$K_1(\mathcal{MC}_\psi) \cong 0$.
Hence the C$^*$-algebra $\mathcal{MC}_\psi$ is not
isomorphic to the Cuntz algebra $\mathcal{O}_\infty$.
\end{rem}

\end{document}